\documentclass[a4paper,11pt,twoside]{amsart}
\usepackage[margin=1.2in]{geometry}
\usepackage{mathtools}
\usepackage{amsfonts}
\usepackage{amsthm}
\usepackage{amsmath}
\usepackage{amssymb}
\usepackage{enumerate}
\usepackage{graphicx}
\usepackage{subcaption}
\usepackage{url}

\let\originalleft\left
\let\originalright\right
\renewcommand{\left}{\mathopen{}\mathclose\bgroup\originalleft}
\renewcommand{\right}{\aftergroup\egroup\originalright}
\renewcommand{\Re}{\operatorname{Re}}
\renewcommand{\Im}{\operatorname{Im}}

\newcommand{\meas}{\operatorname{meas}}

\newcommand{\R}{\mathbb{R}}
\newcommand{\N}{\mathbb{N}}
\newcommand{\eps}{\varepsilon}
\renewcommand{\hat}{\widehat}

\newtheorem{theorem}{Theorem}
\newtheorem{lemma}[theorem]{Lemma}

\theoremstyle{remark}
\newtheorem*{remark}{Remark}

\begin{document}
\title{Large deviations of the argument of the Riemann zeta function}
\author{Alexander Dobner}
\date{February 13, 2023}
\begin{abstract}
Let $S(t) = \frac{1}{\pi}\Im \log\zeta\left(\frac{1}{2}+it\right)$. We prove an unconditional lower bound on the measure of the sets $\{t\in [T,2T] \colon S(t) \geq V\}$ for $\sqrt{\log\log T} \leq V \ll \left(\frac{\log T}{\log \log T}\right)^{1/3}$. For $V \leq (\log T)^{1/3-\eps}$ our bound has a Gaussian shape with variance proportional to $\log\log T$. At the endpoint, $V \asymp \left(\frac{\log T}{\log \log T}\right)^{1/3}$, our result implies the best known $\Omega$-theorem for $S(t)$ which is due to Tsang. We also explain how the method breaks down for $V \gg \left(\frac{\log T}{\log \log T}\right)^{1/3}$ given our current knowledge about the zeros of the zeta function. Conditionally on the Riemann hypothesis we extend our results to the range $\sqrt{\log\log T} \leq V \ll \left(\frac{\log T}{\log \log T}\right)^{1/2}$.
\end{abstract}
\maketitle

\section{Introduction}
Let
\[
S(t) \coloneqq \frac{1}{\pi} \Im \log \zeta\left(\frac{1}{2}+i t\right).
\]
In this definition we assume that $\log \zeta(s)$ is defined with branch cuts extending horizontally to the left of each zero and pole of the zeta function, and we choose the branch such that $\log \zeta(2)$ is real. When $\frac{1}{2}+it$ lies on a branch cut, $S(t)$ has a jump discontinuity, and we set $S(t)\coloneqq \lim_{\eps\to 0} \left(S(t+\eps)+S(t-\eps)\right)/2$ in this case. 

The behavior of $S(t)$ is closely connected to the locations of the nontrivial zeros of the zeta function. Let $N(t)$ denote the number of zeta zeros $\rho=\beta+i\gamma$ with $0 < \gamma \leq t$ such that any zeros with $\gamma=t$ are counted with weight $1/2$. The Riemann-von Mangoldt formula states that
\begin{equation} \label{eq:rvm}
N(t) = \frac{t}{2\pi}\log\frac{t}{2\pi} - \frac{t}{2\pi} + \frac{7}{8}+S(t) + O\left(t^{-1}\right), \quad t\geq 1.
\end{equation}
It is a classical result of Backlund (see, e.g., \cite[Thm. 9.4]{titchmarsh}) that $S(t)=O(\log t)$. From this it is clear that the behavior of $N(t)$ is largely controlled by the first two terms in \eqref{eq:rvm}. One should think of $S(t)$ as measuring the error between $N(t)$ and the ``expected'' number of zeta zeros up to height $t$. Hence, large values of $S(t)$ reflect irregularities in the locations of the nontrivial zeta zeros.

\subsection{Main result}

In this paper we consider the question of how often $S(t)$ is large for $t \in [T,2T]$. The distribution of $S(t)$ was first studied by Selberg \cite{selberg1946} who showed that the moments of $S(t)$ on $[T,2T]$ are asymptotic to the moments of a normal distribution with mean zero and variance  $\frac{1}{2\pi^2}\log\log T$. From this one can deduce the \emph{Selberg central limit theorem}, which states that for fixed $\Delta > 0$,
\begin{equation} \label{eq:selbergclt}
\frac{1}{T} \meas \left\{t\in[T,2T] \colon \frac{\pi S(t)}{\sqrt{\frac{1}{2}\log\log T}} \geq \Delta \right\} = \frac{1}{\sqrt{2\pi}} \int_\Delta^\infty e^{-u^2/2}\, du + o(1)
\end{equation}
as $T$ tends to infinity. It is natural then to ask whether this Gaussian behavior persists in a large deviations regime where $\Delta$ is allowed to grow with $T$. We prove the following theorem giving a lower bound on the frequency of such large deviations.

\begin{theorem} \label{Sthm}
There exist constants $\kappa, D > 0$ such that for all $T$ sufficiently large and all $V$ in the range $\sqrt{\log \log T} \leq V \leq \kappa \left(\frac{\log T}{\log \log T}\right)^{\frac{1}{3}}$ we have
\[
\meas\, \left\{t \in [T,2T] \colon S(t) \geq V \right\} \geq T \exp\left(-D \frac{V^2}{\log \frac{\log T}{V^3 \log V}} \right).
\]
The same is true if $S(t)$ is replaced with $-S(t)$.
\end{theorem}

We now provide some commentary on this theorem. The lower endpoint on the range of $V$ in the theorem statement is not significant because for $V$ proportional to $\sqrt{\log\log T}$, the result is superseded by the Selberg central limit theorem. In fact, even in the larger range $V \ll (\log\log T)^{\frac{1}{2}+\frac{1}{10}-\eps}$ a result of Radziwi\l\l{ }\cite[Rmk.~(1)]{radziwill2011} states that
\begin{equation} \label{eq:RadAsymp}
\frac{1}{T}\meas\{t\in[T,2T]\colon \pi S(t) \geq V\} \sim \frac{1}{\sqrt{2\pi}} \int_\Delta^\infty e^{-u^2/2}\, du, \quad T\to\infty
\end{equation}
where $\Delta \coloneqq V/\sqrt{\frac{1}{2}\log\log T}$. In other words, the Gaussian behavior predicted by the Selberg central limit theorem persists in this range. To our knowledge, for deviations larger than the range in Radziwi\l\l's theorem, our Theorem 1 is the first lower bound on the frequency of large values of $S(t)$ in the literature.

At first glance the bound in Theorem 1 looks rather complicated. However, in the range $V \leq (\log T)^{1/3-\eps}$, the bound simplifies to $T \exp(-D_\eps V^2/\log\log T)$ for some constant $D_\eps>0$. Thus, in this range the bound has the shape of Gaussian whose variance is proportional to $\log\log T$. Up to the proportionality factor this matches the variance in the Selberg central limit theorem. For values of $V$ very near the endpoint $\kappa \left(\frac{\log T}{\log \log T}\right)^{\frac{1}{3}}$, the strength of the bound decreases, but it is still nontrivial. Hence, as a corollary of Theorem 1 we find that $S(t) = \Omega_\pm\left(\left(\frac{\log t}{\log\log t}\right)^{\frac{1}{3}}\right)$. This result is not new (it is due to Tsang \cite{tsang1986}) but it is the best known unconditional $\Omega$-theorem for $S(t)$. Tsang's result does not say anything about the frequency of such large values except that they occur in every interval $[T,2T]$. Theorem 1 provides a lower bound of the form $T\exp(-D_0 V^2)$ at this endpoint. 

\subsection{Proof methods}
The proof of Theorem~\ref{Sthm} has four parts. The first is a convolution formula of  Selberg~\cite{selberg1946} which expresses a smoothed version of $S(t)$ in terms of a contribution from zeta zeros off the critical line and a contribution from a Dirichlet polynomial. The second part of the proof is an upper bound on how often the contribution from the zeta zeros is large. To get this bound we use a zero density estimate and a theorem of Tsang on the moments of $S(t+h)-S(t)$ to get some control over the zeta zeros in the horizontal and vertical directions respectively. The third part of the proof is a lower bound on how often the Dirichlet series contribution in the convolution formula is large. This is done via an application of the \emph{resonance method} which was developed by Soundararajan~\cite{soundararajan2008}. Combining these three parts gives a lower bound on the frequency of large values of the smoothed $S(t)$.  In the fourth part of the proof we relate this back to the frequency of large values of $S(t)$.

If we assume the Riemann hypothesis, this eliminates the possibility of a contribution from zeta zeros off the critical line. In this case we are able to prove a large deviations result that holds for a larger range of $V$. We state this result now.

\begin{theorem} \label{Sthm2}
Assuming the Riemann hypothesis holds, there exist constants $\widetilde{\kappa}, \widetilde{D} > 0$ such that for all $T$ sufficiently large and all $V$ in the range $\sqrt{\log \log T} \leq V \leq \widetilde{\kappa} \left(\frac{\log T}{\log \log T}\right)^{\frac{1}{2}}$ we have
\[
\meas\, \left\{t \in [T,2T] \colon S(t) \geq V \right\} \geq T \exp\left(-\widetilde{D} \frac{V^2}{\log \frac{\log T}{V^2 \log V}} \right).
\]
The same is true if $S(t)$ is replaced with $-S(t)$.
\end{theorem}

The key difference between Theorem 1 and Theorem 2 is the change of exponent from $\frac{1}{3}$ to $\frac{1}{2}$ in the upper endpoint of the range of $V$. In Section 2 we provide a more detailed outline of our proof methods, and we explain how the smaller exponent $\frac{1}{3}$ comes about in the unconditional result.

\subsection{Related work}
There is a natural connection between the behavior of $S(t)$ and the behavior of the magnitude of $\zeta(s)$ on the critical line because
\[
S(t) = \frac{1}{\pi}\Im \log \zeta(\tfrac{1}{2}+it) \text{ whereas } \log |\zeta(\tfrac{1}{2}+it)|=\Re \log\zeta(\tfrac{1}{2}+it).
\]
A great deal of research has been done on large values of both of these functions. The most notable open problem in this area is the Lindel\"{o}f hypothesis, which states that $\log |\zeta(\tfrac{1}{2}+it)| = o(\log t)$. On the Riemann hypothesis, the slightly stronger $\log |\zeta(\tfrac{1}{2}+it)| = O(\log t/\log\log t)$ is known, and similarly for $S(t)$ (see, e.g., \cite[Thms. 14.13, 14.14(A)]{titchmarsh}). There is reason to believe that these upper bounds are still rather far from the truth. Indeed, probabilistic heuristics suggest that the true largest values of $\log |\zeta(\tfrac{1}{2}+it)|$ and $S(t)$ are of size $\asymp \sqrt{\log t\log\log t}$ (see \cite{farmer2007gh} for arguments for and against this).

Just like for $S(t)$ the Selberg central limit theorem is known to hold for $\log |\zeta(\frac{1}{2}+it)|$, so one can again study the frequency of large deviations. The analogue of Radziwi\l\l's{ }asymptotic result mentioned above still holds (see \cite[Thm. 1]{radziwill2011}). That is, for $V \ll (\log\log T)^{\tfrac{1}{2}+\tfrac{1}{10}-\eps}$,
\begin{equation} \label{eq:RadAsymp2}
\frac{1}{T}\meas\{t\in[T,2T]\colon \log|\zeta(\tfrac{1}{2}+it)| \geq V\} \sim \frac{1}{\sqrt{2\pi}} \int_\Delta^\infty e^{-u^2/2}\, du, \quad T\to\infty
\end{equation}
where $\Delta \coloneqq V/\sqrt{\frac{1}{2}\log\log T}$. In the same paper Radziwi\l\l{ }conjectures that this asymptotic should continue to hold in the entire range $V = o(\log\log T)$. Borrowing a term from the probability theory literature, we refer to this range as the \emph{moderate deviations} regime.

When $V=\alpha \log\log T$ for fixed $\alpha>0$, the frequency of large deviations of $\log |\zeta(\tfrac{1}{2}+it)|$ is of particular interest because these values are expected to be the main contributor to the $2\alpha$\textsuperscript{th} moment of the zeta function (see \cite{soundararajan2009} and also \cite{radziwill2011} for a discussion of this). Radziwi\l\l{ } gives a natural conjecture (see \cite[Conj. 2]{radziwill2011}) on the frequency of these large deviations as well. The asymptotic in this case is the same as in the moderate deviations case except for a constant factor depending on $\alpha$. There are some partial results in the literature in this direction. In particular, for $0<\alpha<2$, Arguin and Bailey \cite{arguin2023b} have recently proved an upper bound on the frequency of large deviations which is nearly optimal in that it matches the conjecture except for the constant factor. Soundararajan \cite{soundararajan2009} has proved similar upper bounds that hold for a much larger range of $V$ but which are not quite as sharp and which are conditional on the Riemann hypothesis.

In a similar vein to the theorems in this paper, Soundararajan \cite{soundararajan2008} has also proved a \emph{lower} bound on the frequency of large deviations of $\log |\zeta(\frac{1}{2}+it)|$. Unlike for $S(t)$, the possibility of zeta zeros off the critical line does not seem to affect the strength of the results one can prove for $\log |\zeta(\frac{1}{2}+it)|$ (see the remark at the end the next section for more on this). Consequently, the strength of the bound and the range of $V$ in \cite[Thm. 1]{soundararajan2008} are comparable to our Theorem~2, but Soundararajan's result is unconditional.

With regards to $\Omega$-theorems, it is interesting to note that the best conditional result for $S(t)$ goes even further than the endpoint $(\log T/\log\log T)^{1/2}$. Indeed, on the Riemann Hypothesis it is known that $S(t) = \Omega_\pm\left(\sqrt{\frac{\log t \log\log\log t}{\log\log t}}\right)$. The extra $\sqrt{\log\log\log t}$ factor comes from a recent breakthrough of Bondarenko and Seip \cite{bondarenko2018s}. Unfortunately, Bondarenko and Seip's method only produces these large values on a very sparse set. The fact that Bondarenko and Seip's method can be used to get large values of either sign was noted by Chirre and Mahatab \cite{chirre2021m}.

\section{Further details on proof methods} \label{sec:heuristics}
The standard method for studying the distribution of $\log \zeta$ on the critical line is to use an \emph{explicit formula} to express $\log \zeta(\tfrac{1}{2}+it)$ in terms of a sum over primes (more specifically a Dirichlet polynomial) plus a sum over the zeros of the zeta function near $\tfrac{1}{2}+it$. In this paper we will use an explicit formula for a \emph{smoothing} of $\log \zeta$. This approach was first used by Selberg in his work on large values of $S(t)$ (see \cite[Section 7]{selberg1946}). Roughly speaking, if one convolves $S(t)$ with an approximation to the identity of width $\lambda^{-1}$ (where $\lambda$ is a parameter) then the resulting formula looks like
\begin{equation} \label{eq:Sheuristic}
\text{smoothed $S(t)$} \approx \frac{1}{\pi} \Im \sum_{\log p \ll \lambda} \frac{1}{\sqrt{p}}p^{-it} + \sum_{\substack{\beta > 1/2 \\ |\gamma -t| \ll \lambda^{-1}}} G_{t,\lambda}(\rho)
\end{equation}
where the second sum is over zeta zeros $\rho = \beta+i\gamma$, and $G_{t,\lambda}$ is some function whose definition is not so important for now. A key feature of this formula is that only the zeta zeros to the right of the critical line contribute, and (as we shall see) the contribution of zeros close to the critical line will be rather small. Since large values of $S(t)$ should occur at least as often as large values of a smoothing of $S(t)$ (this will be formalized in Section 6),  it will suffice to consider the quantity on the right-hand side of~\eqref{eq:Sheuristic}.

Note that if we assume the Riemann hypothesis then \eqref{eq:Sheuristic} has no contribution coming from zeta zeros. In this case the problem of finding large values of $S(t)$ is reduced to finding large values of the imaginary part of a certain Dirichlet polynomial. If one just wants to prove $\Omega$-results, there are several methods appearing in the literature: computing moments of the Dirichlet polynomial (see Tsang \cite{tsang1986}), using Dirichlet's approximation theorem (see Montgomery \cite{montgomery1977}), or using the resonance method (see Bondarenko and Seip \cite{bondarenko2018s}). The resonance method also allows one to get a lower bound on the frequency of large values rather easily, so this is the approach we will take in this paper.

If we do not assume the Riemann hypothesis, we must account for the possibility of a contribution from zeta zeros off the critical line. Tsang's proof of his unconditional $\Omega$-theorem for $S(t)$ involves comparing high moments of the Dirichlet polynomial contribution in \eqref{eq:Sheuristic} with high moments of the zeros contribution and then applying a clever lemma to deduce the existence of large values of $S(t)$.  Our approach will be a bit more direct. Given some $V$ we first prove that the magnitude of the contribution of zeta zeros in the convolution formula is less than $V$ for all $t \in [T,2T]$ outside of some exceptional set of small measure. Next we show that the Dirichlet polynomial contribution is larger than $2V$ on a set whose measure is much larger than the exceptional set. From this we can deduce a bound on how often the sum of these two contributions is greater than $V$. 


We now perform a heuristic analysis of the Dirichlet polynomial contribution and the zeta zeros contribution in \eqref{eq:Sheuristic}. We start with the Dirichlet polynomial. Given any $V$ (which we think of as growing slowly with $T$) we would like to know how often the Dirichlet polynomial is greater than $2V$. First observe that the terms in the sum coming from primes smaller than $V$ (say) are inconsequential because they can contribute at most $\sum_{p\leq V} \frac{1}{\sqrt{p}} = o(V)$ to the total. This immediately implies that we must take $\lambda \gg \log V$ if the sum is ever to be larger than $V$. For the remaining terms, we can model the sum by a random sum since the phases $p^{-it}$ behave roughly like i.i.d. random variables with mean zero. The variance of this random sum is proportional to  
\[
\sum_{\log V \leq \log p \ll \lambda} \frac{1}{p} = \log \lambda - \log \log V +O(1)
\]
 by Mertens' theorem. So if the distribution of the random sum is approximately Gaussian (even for large deviations), then this model predicts that our deterministic sum is greater than $2V$ on a subset of $[T,2T]$ whose measure is at least
\begin{equation} \label{eq:primesmeasbd}
T\exp\left(-C \frac{V^2}{\log\frac{\lambda}{\log V}}\right).
\end{equation}
for some $C>0$. We shall see in Section~\ref{sec:DirichletCont} that it is possible to prove a bound of precisely this form (assuming that $\lambda$ and $V$ lie in some appropriate ranges) using Soundararajan's resonance method. The fact that we require $V \ll (\log T/\log \log T)^{1/2}$ in Theorem 2 comes from the limitations of what we are able to prove using the resonance method.

Next we analyze the contribution of the zeta zeros in \eqref{eq:Sheuristic}. The average spacing of zeta zeros near height $T$ is $\frac{2\pi}{\log T}$, so we expect that usually there are about $O(\lambda^{-1}\log T)$ terms in this sum. Later we will see that the magnitude of each summand $G_{t,\lambda}(\rho)$ is roughly $O(\lambda^2 (\beta-\frac{1}{2})^2)$. Hence, if we let $\theta_t \coloneqq \max (\beta-\frac{1}{2})$ where the maximum is taken over all the zeros in the sum, then by the triangle inequality we find
\begin{equation} \label{eq:zerosbd}
\sum_{\substack{\beta > 1/2 \\ |\gamma -t| \ll \lambda^{-1}}} G_{t,\lambda}(\rho) \ll \lambda^2 \sum_{\substack{\beta > 1/2 \\ |\gamma -t| \ll \lambda^{-1}}} (\beta-\frac{1}{2})^2 \ll \lambda \theta_t^2 \log T.
\end{equation}
If this bound holds for a given $t\in[T, 2T]$, then the contribution of the zeros is guaranteed to be less than $V$ whenever
\begin{equation} \label{eq:thetatbd}
\theta_t \ll \left(\frac{V}{\lambda\log T}\right)^{\frac{1}{2}}
\end{equation}
for some sufficiently small implicit constant. Hence, we would like to bound the measure of the exceptional set of $t$ values where \eqref{eq:thetatbd} fails to hold. To do this we need a bound on the number of zeta zeros lying a specified distance off the critical line. Such bounds are known as \emph{zero density estimates}. We will see in Lemma~\ref{EVlemma} that by using a classical zero density estimate we can bound the measure of our exceptional set by
\begin{equation} \label{eq:zerosmeasbd}
T \exp(10 \lambda) \exp\left(-c \left(\frac{V\log T}{\lambda}\right)^{\frac{1}{2}}\right)
\end{equation}
where $c$ is some positive absolute constant.

We now explain how the exponent $\frac{1}{3}$ arises in Theorem 1. Given a particular choice of $V$, we would like to select a value of the parameter $\lambda$ so that the lower bound \eqref{eq:primesmeasbd} is larger than the upper bound \eqref{eq:zerosmeasbd}. This is only possible if $V$ is not too large. To find the endpoint, suppose that $(\log T)^{1/10} \leq V \leq \log T.$ We saw in our analysis of the Dirichlet series that we must choose $\lambda \gg \log V \gg \log \log T$. In light of this, write $\lambda = A_{V,T} \log\log T$ where $A_{V,T} \gg 1$ is some quantity which may grow with $V$ or $T$ in some way. The lower bound \eqref{eq:primesmeasbd} becomes
\[
T\exp \left(-C' \frac{V^2}{\log A_{V,T}}\right)
\]
whereas the upper bound \eqref{eq:zerosmeasbd} becomes
\[
T \exp(10 A_{V,T}\log\log T) \exp\left(-c'\left(\frac{V \log T}{A_{V,T}\log\log T}\right)^{\tfrac{1}{2}}\right).
\]

A short calculation now shows that the crossover point between these two bounds occurs when $V \asymp \left(\frac{\log A_{V,T}}{A_{V,T}}\frac{\log T}{\log\log T}\right)^{\frac{1}{3}}$. Hence, taking $A_{V,T}\asymp 1$ is the best choice for making this crossover occur at the largest possible value of $V$. This is where the endpoint $V \leq \kappa \left(\frac{\log T}{\log\log T}\right)^{\frac{1}{3}}$ comes from in Theorem 1. For values of $V$ less than this endpoint, one needs to choose a larger value of $\lambda$ in order to get the right measure bound in the conclusion of Theorem~\ref{Sthm}. For example, consider the case where $V = (\log T)^{\frac{1}{3}-\eps}$ for some $0 < \eps < \frac{1}{3}$. Letting $\lambda = (\log T)^\mu$ for some small $\mu$ depending on $\eps$, one sees that the lower bound \eqref{eq:primesmeasbd} is at least
\[
T \exp\left(-D_\eps \frac{V^2}{\log \log T}\right)
\]
for some constant $D_\eps>0$ depending on $\eps$ (note that this bound now contains the correct variance factor $\log\log T$). One can then check that the upper bound \eqref{eq:zerosmeasbd} is much smaller than this as long as $\mu$ is chosen to be small enough.

\begin{remark}
The most natural target for improving Theorem 1 would be to get a better bound on the contribution of the zeta zeros. In the next section we will see a precise statement of Selberg's convolution formula for $\log \zeta$, and then it will be clear exactly what the contribution of each individual zero looks like. Interestingly, if we apply the methods we have described above to $\Re \log \zeta(\tfrac{1}{2}+it)$ rather than $S(t)$ it is possible to avoid most of the difficulties coming from the contribution of the zeros. The reason for this is that if one chooses an appropriate kernel function in Selberg's formula, one can ensure that the \emph{real} parts of all the terms coming from zeta zeros which are slightly to the right of the critical line are positive. This means that such terms are not detrimental for finding large values (see Tsang \cite{tsang1993} for an example of this idea). Carrying out this argument we can reproduce Soundararajan's large deviations result mentioned above which is valid in the range $V \ll (\log T/\log\log T)^{1/2}$. Unfortunately, when one considers the \emph{imaginary} part, as we will in this paper, then each zero may contribute positively or negatively depending on its height, and so we are forced to apply the triangle inequality and bound the magnitude of each term.
\end{remark}

\section{Convolution formula} \label{sec:Conv}
We now begin the full proof Theorems 1 and 2. From here onward we assume $t \in [T,2T]$ where $T$ is large. We start by quoting a lemma given in \cite{tsang1986} which is a general form of the convolution formula for $\log \zeta$ discovered by Selberg. 
\begin{lemma} \label{tsanglemma}
Suppose $1/2 \leq \sigma < 1$ and let $K(x+iy)$ be an analytic function in the horizontal strip $\sigma-2 \leq y \leq 0$ satisfying the growth condition
\[
W(x) \coloneqq \sup_{\sigma-2 \leq y \leq 0} |K(x+iy)| = O\left((|x| \log^2 |x|)^{-1}\right)
\]
as $|x| \to \infty$. For any $t \neq 0$, we have
\begin{multline}
\int_{-\infty}^{\infty} \log \zeta\left(\sigma+i(t+u)\right)K(u)\, du = \\
\sum_{n=2}^{\infty} \hat{K}(\log n) \frac{\Lambda(n)}{\log n}  n^{-\sigma-it}
 + 2\pi \sum_{\substack{\rho = \beta+i \gamma \\ \beta>\sigma}} \int_0^{\beta-\sigma} K(\gamma - t - i\alpha )\, d\alpha + O(W(t)).
\end{multline}
\end{lemma}
The second summation above runs over all zeta zeros lying strictly to the right of the line $\Re(s) = \sigma$. The convention used here for the Fourier transform is $\hat{K}(x) \coloneqq \int K(u) e^{-i x u}\, du$.

We will be applying the lemma with a kernel $K_\lambda(x)=\lambda \omega(\lambda x)$ where $\omega$ is a fixed ``weight'' function and $\lambda \in [1, \log T]$ is a parameter (to be chosen later possibly depending on $T$ and $V$). The precise choice of $\omega$ is not so important, but for concreteness we select the function
\[
\omega(u) \coloneqq \frac{1}{2\pi} \left(\frac{\sin(u/2)}{u/2}\right)^2
\]
whose Fourier transform is $\hat{\omega}(x) = \max(0,1-|x|)$.
\[
K_{\lambda}(z) \coloneqq \lambda \omega\left(\lambda z\right).
\]
With this choice of $\omega$ one can verify that $K_\lambda$ has the following properties:
\begin{enumerate}[\quad (i)]
\item $K_\lambda(u) \geq 0$ for $u \in \R$,
\item $\int K_\lambda(u) \, dx = 1$,
\item $\Im K_{\lambda}(u+iv) \ll \frac{\lambda^2 |v| e^{\lambda |v|}}{1+(\lambda u)^2+(\lambda v)^2}$,
\item 
$\widehat{K_\lambda}(x) = \hat{\omega}\left(\frac{x}{\lambda}\right)$. \label{kerneltransform}
\end{enumerate}
(For a proof of property (iii) see \cite[pg. 52]{selberg1946}.)

Applying Selberg's convolution formula with $K_\lambda$, we now prove that a short average of $S(t)$ can be expressed in terms of a Dirichlet polynomial and some contribution from zeta zeros.
\begin{lemma} \label{Sconvlemma} Suppose $1 \leq \lambda \leq \log T$ and $T$ is large. For any $t\in [T,2T]$,
\[
\int_{-\log T}^{\log T} S(t+u) K_\lambda(u)\, du = \frac{1}{\pi} \Im D_\lambda(t) + Z_{\lambda}(t) + O\left(1\right)
\]
where
\begin{align*}
D_\lambda(t) &= \sum_{k=2}^{\infty} \hat{\omega}\left(\frac{\log k}{\lambda}\right) \frac{\Lambda(k)}{(\log k) \sqrt{k}}\,  k^{-it}, \\
Z_\lambda(t) &=  2 \Im \sum_{\beta>\frac{1}{2}} \int_0^{\beta-\frac{1}{2}} K_\lambda(\gamma - t - i\alpha )\, d\alpha.
\end{align*}
\end{lemma}
\begin{proof}
Applying Lemma \ref{tsanglemma} with the kernel $K_\lambda$ and $\sigma=1/2$, we see that
\begin{multline*}
\int_{-\infty}^{\infty} \log \zeta\left(\frac{1}{2}+i(t+u)\right)K_{\lambda}(u)\, du =
\sum_{n=2}^{\infty} \hat{\omega}\left(\frac{\log n}{\lambda}\right) \frac{\Lambda(n)}{\log n} n^{-\frac{1}{2}-it} \\
 + 2\pi \sum_{\substack{\rho = \beta+i \gamma \\ \beta>\frac{1}{2}}} \int_0^{\beta-\frac{1}{2}} K_{\lambda}(\gamma - t - i\alpha )\, d\alpha + O\left(W(t)\right).
\end{multline*}
where
\[
W(t) = \sup_{-3/2 \leq y \leq 0} |K_\lambda(t + iy)|.
\]

Taking imaginary parts and dividing through by $\pi$ we get
\begin{equation} \label{eq:convKab}
\int_{-\infty}^{\infty} S(t+u) K_{\lambda}(u)\, du = \frac{1}{\pi} \Im D_\lambda(t) + Z_{\lambda}(t) + O\left(W(t)\right).
\end{equation}

Using the assumptions $t \in [T, 2T]$ and $\lambda \leq \log T$ and the inequality $|\sin(a+ib)| \ll e^{|b|}$, we see that
\[
W(t) \ll \log T \frac{e^{\frac{3}{2}\log T}}{T^2} \ll 1.
\]

All that remains is to show that the tails of the integral in \eqref{eq:convKab} are small. Using the classical bound $S(t) = O(\log(|t|+2))$ and the bound $K_\lambda(u) \ll \lambda^{-1} u^{-2}$, we see
\[
\int_{|u|>\log T} S(t+u) K_{\lambda}(u)\, du \ll \int_{\log T}^\infty (\log T + \log u) \lambda^{-1} u^{-2}\, du \ll 1.
\]
Hence, the tails may be absorbed into the error term.
\end{proof}

\section{Contribution of the zeros} \label{sec:ZerosCont}
Our goal in this section will be to prove that the zeros contribution $Z_\lambda(t)$ is small outside of some exceptional set whose measure we can bound. 

The average gap between consecutive zeta zeros with imaginary parts in $[T,2T]$ is asymptotic to $\frac{2\pi}{\log T}$, so we expect that in most short subintervals $[t,t+h]$ there should be approximately $\frac{h}{2\pi}\log T$ zeros. We will need a quantitative estimate on how often a short interval contains ``too many'' zeros.

\begin{lemma} \label{zeroslemma}
There exists an absolute constant $d > 0$ such that for all $T \geq 3$ and $1/\log T \leq h \leq 2$ the bound
\begin{equation} \label{eq:shortintbd}
\meas \left\{t \in [T,2T] \colon N(t+h)-N(t) \geq h \log T \right\} \ll T \exp\left(-d h \log T\right)
\end{equation}
holds uniformly.
\end{lemma}

We could not locate a reference for such a bound in the literature, so we supply a proof. The main input in the proof is the following moment estimate due to Tsang.

\begin{lemma}[Tsang {\cite[Thm. 4]{tsang1986}}] \label{momentlemma}
For any positive integer $k$ and any $0 < h < 1$,
\begin{multline}
\int_T^{2T} (S(t+h)-S(t))^{2k}\, dt \\
= A_k T \left(\log(2+h \log T)\right)^k+O\left(T(ck)^k\left(k^k+\left(\log(2+h\log T)\right)^{k-1/2}\right)\right)
\end{multline}
where $c$ is a positive constant and
\[
A_k \coloneqq \frac{(2k)!}{2^k \pi^{2k}k!}.
\]
\end{lemma}

\begin{proof}[Proof of Lemma~\ref{zeroslemma}]
We may assume $h \log T$ is larger than some fixed constant to be chosen later (and hence $T$ is large as well) because for small $h \log T$ the conclusion of the lemma is trivial.

From the Riemann-von Mangoldt formula one sees that
\[
N(t+h)-N(t) = \frac{h}{2\pi} \log \frac{t}{2\pi} + S(t+h)-S(t) + O(T^{-1})
\]
for any $t \in [T,2T]$. Hence,
\[
\left\{t \in [T,2T] \colon N(t+h)-N(t) \geq h \log T\right\} \subseteq \left\{t \in [T,2T] \colon |S(t+h)-S(t)| \geq \frac{1}{2} h \log T\right\}
\]
for all sufficiently large $T$. Let $G$ denote the set on the right-hand side. It suffices to show that
\begin{equation} \label{eq:Gbd}
\meas G \ll T \exp(-d h \log T).
\end{equation}
To do this we use Markov's inequality and the moment estimate from Lemma~\ref{momentlemma}.

Note that since $h \geq 1/\log T$ we have
\[
\log(2 + h \log T) \leq 1+h\log T \leq 2h \log T.
\]
Inserting this into Lemma~\ref{momentlemma} and applying Stirling's approximation to estimate $A_k$ we get an upper bound
\[
\int_T^{2T} (S(t+h)-S(t))^{2k}\, dt \leq T (c' k)^k ((h \log T)^k+k^k)
\]
where $c'>0$ is some absolute constant and $k$ can be taken to be any positive integer. From Markov's inequality we have
\[
\meas G \leq \frac{T (c' k)^k ((h \log T)^k+k^k)}{\left(\frac{1}{2}h \log T \right)^{2k}}.
\]
If $h \log T \geq 8 c'$, then applying this bound with 
\[
k = \left\lfloor \frac{1}{8c'} h\log T\right\rfloor \geq 1
\]
gives \eqref{eq:Gbd}. Hence we have proved the lemma for $h \log T$ larger than some absolute constant. 
\end{proof}

We will also require a \emph{zero density estimate} on the number of zeros off the critical line.  Let $N(\sigma, T)$ denote the number of zeta zeros $\rho = \beta+i\gamma$ for which $\beta > \sigma$ and $0 \leq \gamma \leq T$. For our purposes the following classical result will suffice.

\begin{lemma}[Zero density estimate] \label{zerodensitylemma}
\[
N(\sigma, T) \ll T^{\frac{3}{2}-\sigma} \log^5 T
\]
\end{lemma}
\begin{proof}
See Titchmarsh \cite[Theorem 9.19(A)]{titchmarsh}.
\end{proof}

We can now bound the contribution of $Z_\lambda(t)$ outside of an exceptional set.
\begin{lemma} \label{EVlemma}
Suppose $\log\log T \leq \lambda \leq \log T/2$ and $2 \leq V \leq \log T/\lambda$. Let 
\[
E_V \coloneqq \{t \in [T,2T] \colon |Z_\lambda(t)| > V\}.
\]
For all $T$ sufficiently large,
\begin{equation} \label{eq:exceptionalbd}
\meas E_V \ll T \exp(10\lambda)\exp\left(-d'\left(\frac{V \log T}{\lambda}\right)^{1/2}\right) 
\end{equation}
for some absolute constant $d' > 0$.
\end{lemma}
\begin{proof}
First we note that
\begin{align*}
Z_\lambda(t) &=  2\Im \sum_{\beta>\frac{1}{2}} \int_0^{\beta-\frac{1}{2}} K_\lambda(\gamma - t - i\alpha )\, d\alpha \\
&\ll \sum_{\beta>\frac{1}{2}} \int_0^{\beta-\frac{1}{2}} \frac{\lambda^2 \alpha e^{\lambda \alpha}}{1+(\lambda (\gamma-t))^2}\, d\alpha \\
&\ll \sum_{\beta > \frac{1}{2}}  \frac{\lambda^2 \left(\beta-\frac{1}{2}\right)^2 e^{\lambda \left(\beta-\frac{1}{2}\right)}}{1+(\lambda (\gamma-t))^2},
\end{align*}
where the second inequality can be deduced from the fact that the integrand is increasing in $\alpha$. In the final sum, each zero $\rho = \beta+i\gamma$ contributes a certain amount depending on the horizontal distance $\beta - 1/2$ and the vertical distance $|\gamma-t|$. We will split up the sum depending on whether $|\gamma - t|$ is small or large. 

To bound the sum over the zeros where $|\gamma-t|$ is small, we will need to assume that there are not ``too many'' zeros in short intervals near height $t$. To make this precise, we define an exceptional set
\[
E' \coloneqq \left\{t \in [T,2T]\  \colon 
\begin{array}{r} N(t+ \lambda^{-1} (k+1)) - N(t+\lambda^{-1} k) \geq \lambda^{-1} \log T \\
 \text{ for some $k = -\left\lfloor e^{2\lambda} \right\rfloor, \ldots, \left\lfloor e^{2\lambda} \right\rfloor-1$}
 \end{array}\right\}.
\]
The measure of $E'$ can be bounded using Lemma~\ref{zeroslemma}. Indeed, note that $E'$ may covered by $2\left\lfloor e^{2\lambda}\right\rfloor$ translates of the set
\[
\left\{t \in [T,2T] \colon N(t+\lambda^{-1})-N(t) \geq \lambda^{-1} \log T \right\}
\]
in addition to some small sets of length $\ll \lambda^{-1} e^{2\lambda}$ near the endpoints $T$ and $2T$. So by the union bound and Lemma~\ref{zeroslemma} we see that
\begin{align*}
\meas E' &\ll \lambda^{-1} e^{2\lambda} + T \exp(2\lambda) \exp(-d \lambda^{-1} \log T) \\
&\ll  T \exp(2\lambda) \exp(-d \lambda^{-1} \log T).
\end{align*}
Hence, the measure of $E'$ is less than the right-hand side of \eqref{eq:exceptionalbd} (as long as $d'$ is chosen to be small enough later) by our assumption that $V \leq \log T/\lambda$. This means that it will suffice for us to bound $Z_\lambda(t)$ for $t \in [T,2T] \setminus E'$ from now on. 

Fix a particular $t \in [T,2T] \setminus E'$ and define intervals $I_k \coloneqq (t+\lambda^{-1}k, t+\lambda^{-1}(k+1)]$ for $k = -\left\lfloor e^{2\lambda}\right\rfloor, \ldots, \left\lfloor e^{2\lambda}\right\rfloor-1$. Also define 
\[
\theta_t \coloneqq \max_{\substack{\rho = \beta+i \gamma \\ |\gamma-t| \leq e^{2\lambda}}} (\beta-1/2).
\]
Any zeta zero whose imaginary part lies in one of the intervals $I_k$ will have real part at most $\frac{1}{2}+\theta_t$. Therefore,
\[
\sum_{\beta > \frac{1}{2}}  \frac{\lambda^2 \left(\beta-\frac{1}{2}\right)^2 e^{\lambda \left(\beta-\frac{1}{2}\right)}}{1+(\lambda (\gamma-t))^2} \leq \sum_{k} \sum_{\gamma \in I_k} \frac{\lambda^2 \theta_t^2 e^{\lambda \theta_t}}{1+(\lambda (\gamma-t))^2} + \sum_{|\gamma - t| \geq \lambda^{-1}\left\lfloor e^{2\lambda}\right\rfloor} \frac{ \lambda^2 (\frac{1}{2})^2 e^{\lambda/2}}{1+(\lambda (\gamma-t))^2}
\]
where we have trivially bounded the real part of the zeta zeros in last sum by $1$. 

Since $t$ is not in the exceptional set $E'$, each interval $I_k$ contains the imaginary parts of at most $\lambda^{-1} \log T$ zeta zeros, and so
\[
\sum_{k} \sum_{\gamma \in I_k} \frac{\lambda^2 \theta_t^2 e^{\lambda \theta_t}}{1+(\lambda (\gamma-t))^2} \ll
\sum_{k} \frac{(\lambda^2 \theta_t^2 e^{\lambda \theta_t})(\lambda^{-1} \log T)}{1+(\min(|k|-1,0))^2} \ll \lambda \theta_t^2 e^{\lambda \theta_t} \log T.
\]

For the remaining sum (where $|\gamma-t|$ is large) we apply the classical bound $N(t'+1)-N(t') = O(\log t')$. We get
\begin{align*}
\sum_{|\gamma - t| \geq \lambda^{-1}\left\lfloor e^{2\lambda}\right\rfloor} \frac{ \lambda^2 (\frac{1}{2})^2 e^{\lambda/2}}{1+(\lambda (\gamma-t))^2} 
&\ll e^{\lambda/2} \sum_{|\gamma - t| \geq \lambda^{-1}\left\lfloor e^{2\lambda}\right\rfloor} \frac{1}{(\gamma - t)^2} \\
&\ll e^{\lambda/2} \sum_{k=\left\lfloor \frac{1}{2} \lambda^{-1} e^{2\lambda} \right\rfloor}^\infty \frac{1}{k^2} \log(T+k) \\
&\ll e^{\lambda/2} \frac{\lambda (\log T+\lambda)}{e^{2\lambda}}.
\end{align*}
The quantity in the last line tends to zero as $T$ goes to infinity because of the assumption that $\lambda \geq \log\log T$.  Hence if $T$ is sufficiently large we may conclude that the magnitude of this whole sum is less than $1$ say.

Putting together the bounds above we conclude for all $t \in [T,2T] \setminus E'$,
\[
|Z_\lambda(t)| \leq C \lambda \theta_t^2 e^{\lambda \theta_t} \log T + 1
\]
for some absolute constant $C >0$. 

Suppose now that $\theta_t \leq d' \left(\frac{V}{\lambda \log T}\right)^{1/2}$ for a small absolute constant $d' > 0$. Then 
\begin{align*}
Z_\lambda(t) \leq C \lambda \left(d'^2 \frac{V}{\lambda \log T} \right) e^{d' \left(\frac{V \lambda}{\log T}\right)^{1/2}} \log T + 1 \leq C d'^2  e^{d'} V + 1.
\end{align*}
Choosing $d'$ to be sufficiently small (depending on $C$) implies $|Z_\lambda(t)| \leq V$.

We now use the zero density estimate to show that $\theta_t$ is not larger than $d' \left(\frac{V}{\lambda \log T}\right)^{1/2}$ very often. Indeed, by Lemma~\ref{zerodensitylemma} we see that
\[
N\left(\frac{1}{2} + d' \left(\frac{V}{\lambda \log T}\right)^{1/2}, 3T \right) \ll T \log^5 T \exp\left(-d' \left(\frac{V \log T}{\lambda}\right)^{1/2}\right).
\]
This implies that $\theta_t > d' \left(\frac{V}{\lambda \log T}\right)^{1/2}$ on a set of measure
\[
\ll  T \log^5 T \exp(2\lambda) \exp\left(-d' \left(\frac{V \log T}{\lambda}\right)^{1/2}\right).
\]
The result follows.
\end{proof}

\section{Contribution of the Dirichlet polynomial} \label{sec:DirichletCont}

Recall that
\[
D_\lambda(t) = \sum_k \hat{\omega}\left(\frac{\log k}{\lambda}\right) \frac{\Lambda(k)}{(\log k)\sqrt{k}} k^{-it}.
\]
In this section we will prove a rigorous version of the probabilistic heuristics discussed in Section~\ref{sec:heuristics}. We do not strive for the most general lemma possible but rather one which is sufficient for our purposes.

\begin{lemma} \label{FVlemma}
Let $\theta \in [0, 2\pi)$ be an angle. Suppose $40 \log\log T \leq \lambda \leq \frac{1}{40}\log T$ and $\sqrt{\log \lambda} \leq V \leq \frac{1}{20}\left(\frac{\log T}{\lambda}\right)^{\frac{1}{2}}$. Let
\[
F_V \coloneqq \left\{t \in [T,2T] \colon \Re \left(e^{i\theta} D_\lambda(t) \right) \geq V\right\}.
\]
Then for all $T$ sufficiently large,
\[
\meas F_V \geq T \exp\left(-500\frac{V^2}{\log \frac{\lambda}{\log V}}\right).
\]
\end{lemma}

To prove the lemma we will use Soundararajan's resonance method which we describe now. Let
\[
R(t) \coloneqq \sum_{n \leq N} f(n) n^{-it}
\]
for some arithmetic function $f$ and some length $N$ to be chosen later. This Dirichlet polynomial is the \emph{resonator}. Let $\Phi$ be a smooth function supported in $[1,2]$, such that $0 \leq \Phi \leq 1$ and $\Phi(t) = 1$ for $t\in(5/4,7/4)$. We define the following quantities,
\begin{align*}
M_1(R,T) &\coloneqq \int |R(t)|^2 \Phi(\tfrac{t}{T})\, dt, \\
M_2(R,T) &\coloneqq \int \Re \left(e^{i \theta} D_\lambda(t)\right) |R(t)|^2 \Phi(\tfrac{t}{T})\, dt, \\
M_3(R,T) &\coloneqq \int |R(t)|^4 \Phi(\tfrac{t}{T})\, dt, \\
M_4(T) &\coloneqq \int |D_\lambda(t)|^4 \Phi(\tfrac{t}{T})\, dt.
\end{align*}

For any choice of resonator, clearly we have the inequality
\[
\max_{t \in [T,2T]} \Re \left(e^{i \theta} D_\lambda(t)\right) \geq \frac{M_2(R,T)}{M_1(R,T)}.
\]
Therefore, to show that $\Re (e^{i \theta} D_\lambda(t))$ attains large values, one should try to find a resonator that makes the ratio $M_2(R,T)/M_1(R,T)$ as large as possible. From a probabilistic perspective, we can think of $|R(t)|^2 \Phi(\frac{t}{T})$ as a (non-normalized) probability density.  If $t \in [T,2T]$ is chosen at random according to this density, the inequality above states that the maximum of $\Re (e^{i \theta} D_\lambda(t))$ is greater than or equal to the expected value of this quantity. Since we are free to choose the resonator, we can ``tip the scale'' so that the mass of the probability density is concentrated in places where $\Re (e^{i \theta} D_\lambda(t))$ is likely to be large.

Taking these ideas a step further, we can also get a lower bound on the probability of a large value of $\Re (e^{i \theta} D_\lambda(t))$. Observe that
\begin{align*}
M_2(R,T) &\leq V\, M_1(R,T) + \int_{F_V}\Re\left( e^{i \theta} D_\lambda(t) \right) |R(t)|^2 \Phi\left(\tfrac{t}{T}\right)\, dt \\
&\leq V\, M_1(R,T) + \big(\meas F_V\big)^{\frac{1}{4}} M_4(T)^{\frac{1}{4}} M_3(R,T)^{\frac{1}{2}}
\end{align*}
where the second line follows from the first by two applications of Cauchy's inequality. Now note that if the resonator is chosen such that $M_2(R,T) \geq 2 V\, M_1(R,T)$, then rearranging the inequality above gives
\begin{equation} \label{eq:FVbd}
\meas F_V \gg \frac{M_2(R,T)^4}{M_3(R,T)^2 M_4(T)}
\end{equation}
This bound is what we will use to prove the lemma. As we shall see, to make this bound as strong as possible the most important factor to control is $M_3(R,T)$ (that is, we want $M_3(R,T)$ to be small). Hence, our objective will be as follows: given some $V$ (and $\lambda$ and $T$) we must pick a resonator so that $M_2(R,T) \geq 2V\, M_1(R,T)$ with $M_3(R,T)$ is as small as possible.

To proceed with the proof of the lemma, we will first show how to estimate the integrals $M_1,M_2,M_3,M_4$ using standard techniques for estimating mean values of Dirichlet polynomials. Once this is done we will show how to make a suitable choice of $f$ and $N$.

\subsection{Integral estimates}
If $N \leq T^{1-\eps}$, we are able to estimate $M_1$ and $M_2$ as follows. Note that
\[
M_1(R,T) = \int_{-\infty}^\infty |R(t)|^2 \Phi(\tfrac{t}{T})\, dt = T \sum_{m,n \leq N} f(m)\overline{f(n)}\, \hat{\Phi}(T\log(m/n)).
\]
For the terms where $m \neq n$, we have that $|\log(m/n)| \gg T^{\eps-1}$.  Hence $\hat{\Phi}(T \log(m/n)) \ll_\eps T^{-2}$ (say) since $\hat{\Phi}$ is rapidly decaying. Therefore
\begin{align*}
M_1(R,T) &= T \hat{\Phi}(0) \sum_{n \leq N} |f(n)|^2 + O\left(T^{-1} \sum_{m,n \leq N} |f(m)| |f(n)| \right) \\
&= T\hat{\Phi}(0) (1+O(T^{-1})) \sum_{n\leq N} |f(n)|^2.
\end{align*}

Similarly, for $M_2(R,T)$ we have
\[
M_2(R,T) = T \Re e^{i\theta} \sum_{\substack{m,n \leq N \\ k}} \hat{\omega}\left(\frac{\log k}{\lambda}\right)\frac{\Lambda(k)}{(\log k)\sqrt{k}} f(m)\overline{f(n)}\, \hat{\Phi}(T\log(mk/n)).
\]
If $mk \neq n$ then $|\log(mk/n)| \gg T^{\eps-1}$, so $\hat{\Phi}(T \log(mk/n)) \ll_\eps T^{-2}$ in this case. Hence, the total contribution of these off-diagonal terms (i.e. the terms where $mk \neq n$) is bounded by
\[
T^{-1} \sum_k \hat{\omega}\left(\frac{\log k}{\lambda}\right)\frac{\Lambda(k)}{(\log k)\sqrt{k}} \sum_{m,n \leq N} |f(m)||f(n)|.
\]
By the hypothesis of the lemma $\lambda \leq \frac{1}{40} \log T$, the above sum over $k$ can be trivially bounded above by $e^\lambda \leq T^{\frac{1}{40}}$. Hence we may conclude that
\[
M_2(R,T) = T \hat{\Phi}(0) \Re e^{i\theta} \sum_{\substack{ mk = n \\ m,n \leq N}}  \hat{\omega}\left(\frac{\log k}{\lambda}\right)\frac{\Lambda(k)}{(\log k)\sqrt{k}} f(m)\overline{f(n)} + O\left(T^{\frac{1}{40}} \sum_{n\leq N} |f(n)|^2\right).
\]
The precise power of $T$ in the error term is not significant here. We could get any power saving we want since $\hat{\Phi}$ is rapidly decaying.

To estimate $M_3(R,T)$, the method is essentially the same as above. In this case we must assume $N \leq T^{\frac{1}{2}-\eps}$ to be able to bound the contribution of the off-diagonal terms. We will also make the assumption that the resonator coefficients satisfy $|f(n)| \leq 1$ for all $n$. Given these assumptions, we find that
\begin{align*}
M_3(R,T) &= T \sum_{m,n,m',n' \leq N} f(m)f(n)\overline{f(m')}\overline{f(n)} \,  \hat{\Phi}(T \log \tfrac{mn}{m'n'}) \\
&= T \hat{\Phi}(0) \sum_{\substack{m,n,m',n' \leq N \\ mn=m'n'}} f(m)f(n)\overline{f(m')}\overline{f(n')} + O\left(T^{-3} \sum_{m,n,m',n' \leq N} 1 \right) \\
&= T \hat{\Phi}(0) \sum_{\substack{m,n,m',n' \leq N \\ mn=m'n'}} f(m)f(n)\overline{f(m')}\overline{f(n')} + O(T^{-1}).
\end{align*}
In the second equality we used the bound $\hat{\Phi}(T \log \frac{mn}{m'n'}) \ll  T^{-4}$ when $mn \neq m'n'$, and the bound $|f(n)| \leq 1$. Once again we could have gotten any power saving of $T$ that we wanted in the error term.

Lastly, the same methods can also be used to estimate $M_4$. Using the assumption that $\lambda \leq \frac{1}{40} \log T$ to bound off-diagonal terms one may check that
\[
M_4(T)=\int_{-\infty}^{\infty} |D_\lambda(t)|^4 \Phi\left(\tfrac{t}{T}\right)\, dt = T \hat{\Phi}(0)(1+O(T^{-1})) \sum_{n} |c_n|^2
\]
where $c_n$ denotes the $n$th coefficient of the Dirichlet polynomial $D_\lambda(t)^2$. We only need an upper bound on $M_4$, so we proceed by finding an upper bound for $c_n$. Note that $c_n$ is only nonzero when $n$ is a prime power or a product of two prime powers. It is not hard to check that
\[
c_n \ll 
\begin{cases} 
\frac{\log(r+2)}{r} \frac{1}{\sqrt{p^r}} &\text{if }n=p^r \leq e^{2\lambda}, \\
\frac{1}{rs}\frac{1}{\sqrt{p^r q^s}} &\text{if $n=p^r q^s$ such that $p \neq q$ and $p^r, q^s \leq e^\lambda$}, \\
0 &\text{otherwise.}
 \end{cases}
\]

So
\begin{align*}
\sum_n |c_n|^2 &\ll \sum_{p^r,q^s \leq e^\lambda} \frac{1}{r^2 s^2} \frac{1}{p^r q^s} + \sum_{p^r \leq e^{2\lambda}} \frac{\log^2 (r+2)}{r^2} \frac{1}{p^r} \\
&\ll \left(\sum_{p^r \leq e^\lambda} \frac{1}{r^2} \frac{1}{p^r}\right)^2 +  \sum_{p \leq e^{2\lambda}} \frac{1}{p} \\
&\ll \left(\sum_{p \leq e^\lambda} \frac{1}{p}\right)^2 +  \sum_{p \leq e^{2\lambda}} \frac{1}{p} \\
&\ll (\log \lambda)^2,
\end{align*}
where in the last line we have used the classical bound $\sum_{p\leq x} \frac{1}{p} \ll \log \log x$. We conclude that
\[
M_4(T) \ll T \left(\log \lambda\right)^2.
\]

\subsection{Resonator coefficients}
From now on we set $N = T^{\frac{1}{4}}$ since this is sufficient for the integral estimates discussed above to hold. Given $V$, $\lambda$, and $T$ satisfying the hypotheses of the lemma, we would like to optimize our choice of the coefficient function $f(n)$. As we have stated previously, this means choosing an $f$ such that $M_3(R,T)$ is as small as possible while also ensuring that $M_2(R,T)/M_1(R,T)$ is at least $2V$.

Rather than asserting a complete definition of $f$ right away, for now we just impose some restrictions on $f$ and then we will continue with our analysis assuming these restrictions hold. The restrictions are
\begin{enumerate}[\quad\quad (a)]
\item $|f(n)| \leq 1$ (we already used this in our estimate of $M_3(R,T)$ above),
\item $f$ is a multiplicative function supported on squarefree numbers composed of primes in the set $\mathcal{P} \coloneqq \{P \leq p \leq Q\}$ for some parameters $P, Q \geq 3$,
\item $f(p) = e^{i \theta} |f(p)|$,
\item $ - \frac{\log N}{\log Q} + \frac{2}{\log Q} \sum_{p} (\log p) |f(p)|^2 \leq -5$.
\end{enumerate}
The reasons for each of these restrictions will become clear shortly when we use them to deduce a lower bound for $M_2(R,T)/M_1(R,T)$ and an upper bound for $M_3(R,T)$ each of which takes a relatively simple form.

\subsubsection{Lower bound on $M_2(R,T)/M_1(R,T)$}
From our integral estimates, the main term of the quotient $M_2(R,T)/M_1(R,T)$ is
\begin{equation} \label{eq:Mratio}
\Re \left(e^{i\theta} \sum_{\substack{ mk = n \\ m,n \leq N}}  \hat{\omega}\left(\frac{\log k}{\lambda}\right)\frac{\Lambda(k)}{(\log k)\sqrt{k}} f(m)\overline{f(n)} \right) \bigg\slash \left( \sum_{n\leq N} |f(n)|^2 \right).
\end{equation}

We start by considering the numerator of this ratio. Note that in order for a term in the sum to be nonzero, by restriction (b) it is necessary that $k$ be a prime $p \in \mathcal{P}$ and $p$ be distinct from the primes which divide $m$. In this case we have $\overline{f(n)}=\overline{f(p)f(m)}=e^{-i\theta} |f(p)| \overline{f(m)}$ by restriction (c). So we see that the numerator equals
\[
\sum_{p}  \hat{\omega}\left(\frac{\log p}{\lambda}\right)\frac{|f(p)|}{\sqrt{p}}\sum_{\substack{m \leq N/p \\ p \nmid m}}|f(m)|^2.
\]

We will now approximate the inner sum by removing the cutoff at $N/p$. To bound the error (i.e. the sum of tail terms $m>N/p$) we use Rankin's trick. This trick is the observation that for any sequence $\{a_m\}_{m \in \N}$ of nonnegative reals and $\alpha > 0$, 
\[
\sum_{m > X} a_m \leq X^{-\alpha} \sum_{m=1}^\infty m^\alpha a_m.
\]
Applying this to the situation at hand, the numerator in \eqref{eq:Mratio} is at least
\begin{equation} \label{eq:rankined}
\sum_{p}  \hat{\omega}\left(\frac{\log p}{\lambda}\right)\frac{|f(p)|}{\sqrt{p}}\sum_{\substack{m \\ p \nmid m}}|f(m)|^2 - \sum_{p}  \hat{\omega}\left(\frac{\log p}{\lambda}\right)\frac{|f(p)|}{\sqrt{p}}\left(\frac{p}{N}\right)^\alpha \sum_{\substack{m \\ p \nmid m}}m^\alpha |f(m)|^2
\end{equation}
for any $\alpha > 0$. We will now show that the second double sum in \eqref{eq:rankined} is negligible relative to the first double sum (for an appropriate choice of $\alpha$). For notational convenience, define
\[
\Psi \coloneqq \sum_p \hat{\omega}\left(\frac{\log p}{\lambda}\right) \frac{|f(p)|}{\sqrt{p}}.
\]

Note that the inner sum in the main term of \eqref{eq:rankined} satisfies
\[
\sum_{\substack{m \\ p \nmid m}}|f(m)|^2 = \prod_{\substack{q\text{ prime} \\ q \neq p}} (1+|f(q)|^2) \geq
\frac{1}{2} \prod_{q\text{ prime}} (1+|f(q)|^2)
\]
since $|f(q)| \leq 1$ by restriction (a). So we see that the main term of \eqref{eq:rankined} is
\[
\geq \frac{1}{2} \Psi \prod_{q\text{ prime}} (1+|f(q)|^2).
\]

To upper bound the error term we choose $\alpha = 1/\log Q$. This implies $p^\alpha \leq e$ for $p \in \mathcal{P}$, so the magnitude of the error term is
\begin{align*}
&\leq \Psi \,e N^{-\alpha}  \sum_{m}m^\alpha |f(m)|^2 \\
&= \Psi \, e N^{-\alpha} \prod_{q\text{ prime}}(1+q^\alpha |f(q)|^2).
\end{align*}

Hence the magnitude of the ratio of the error term to the main term is
\begin{align*}
&\leq 2e N^{-\alpha} \prod_{q\text{ prime}} \frac{1+q^\alpha |f(q)|^2}{1+|f(q)|^2} \\
&\leq 2e \exp\left(-\alpha \log N + \sum_{q\in \mathcal{P}} (q^\alpha-1)|f(q)|^2\right).
\end{align*}
Since $q^\alpha-1 \leq 2 \alpha \log q$ for $q \in \mathcal{P}$, this is
\begin{align*}
&\leq 2e \exp\left(-\frac{\log N}{\log Q} + \frac{2}{\log Q} \sum_{q \in \mathcal{P}} (\log q) |f(q)|^2\right) \leq 2e \exp(-5) \leq \frac{1}{10}
\end{align*}
by restriction (d). We conclude that the error term in \eqref{eq:rankined} is small relative to the main term, so the numerator of \eqref{eq:Mratio} is at least
\[
\left(1-\frac{1}{10}\right) \frac{1}{2} \Psi\, \prod_{q\text{ prime}} (1+|f(q)|^2) = \frac{9}{20} \Psi\, \prod_{q\text{ prime}} (1+|f(q)|^2)
\]
The denominator of \eqref{eq:Mratio} is clearly bounded above by $\prod_{q\text{ prime}} (1+|f(q)|^2)$, so the quotient is at least $\frac{9}{20}\Psi$. Using this along with our integral estimates for $M_1(R,T)$ and $M_2(R,T)$ we conclude that
\begin{equation} \label{eq:Mratio2}
\frac{M_2(R,T)}{M_1(R,T)} \geq (1+O(T^{-1})) \frac{9}{20}\Psi +O(T^{-\frac{39}{40}}).
\end{equation}

\subsubsection{Upper bound on $M_3(R,T)$}
Note that
\begin{align*}
\left| \sum_{\substack{m,n,m',n' \leq N \\ mn=m'n'}} f(m)f(n)\overline{f(m')}\overline{f(n')} \right|
&\leq \prod_p (1+4|f(p)|^2+|f(p)|^4) \\
&\leq \exp\left(5  \sum_{p} |f(p)|^2\right)
\end{align*}
where we have again used that$|f(p)| \leq 1$ by restriction (b). By our integral estimate for $M_3(R,T)$ we conclude
\begin{equation} \label{eq:M3bd}
M_3(R,T) \ll T  \exp\left(5  \sum_{p} |f(p)|^2\right).
\end{equation}
So minimizing this quantity amounts to minimizing the sum $\sum_{p} |f(p)|^2.$

\subsubsection{Optimizing parameters} Recall that
\[
\Psi = \sum_p \hat{\omega}\left(\frac{\log p}{\lambda}\right) \frac{|f(p)|}{\sqrt{p}}.
\]
If $f$ is chosen so that $\Psi \geq 5V$, then the lower bound \eqref{eq:Mratio2} implies that $M_2(R,T)/M_1(R,T) \geq 2V$ for all $T$ larger than some absolute constant. This gives us the following optimization problem:
\[
\text{minimize }\sum_p |f(p)|^2, \text{ subject to the constraint }\Psi \geq 5V.
\]

To solve this optimization problem, note that applying Cauchy's inequality to the definition of $\Psi$ gives
\[
\Psi \leq \left(\sum_{p \in \mathcal{P}} |f(p)|^2\right)^\frac{1}{2} \left(\sum_{p \in \mathcal{P}} \hat{\omega}\left(\tfrac{\log p}{\lambda}\right)^2 \frac{1}{p}\right)^\frac{1}{2}.
\]
Hence, the constraint $\Psi \geq 5V$ implies that
\begin{equation} \label{eq:Cauchybd}
\sum_{p \in \mathcal{P}} |f(p)|^2 \geq \frac{25 V^2 }{\sum_{p \in \mathcal{P}}\hat{\omega}\left(\tfrac{\log p}{\lambda}\right)^2 \frac{1}{p}}.
\end{equation}
By the equality case of Cauchy's inequality, this lower bound can be attained by choosing $|f(p)|$ to be a certain fixed scalar multiple of $\hat{\omega}\left(\frac{\log p}{\lambda}\right) \frac{1}{\sqrt{p}}$  for each $p \in \mathcal{P}$. Indeed, we see that taking
\begin{equation} \label{eq:optimalfp}
|f(p)| \coloneqq  \frac{5 V}{\sum_{q \in \mathcal{P}}\hat{\omega}\left(\tfrac{\log q}{\lambda}\right)^2 \frac{1}{q}} \hat{\omega}\left(\frac{\log p}{\lambda}\right) \frac{1}{\sqrt{p}}.
\end{equation}
makes \eqref{eq:Cauchybd} into an equality and also makes $\Psi = 5V$.

Combining everything that we have seen so far, as long as all the restrictions (a),(b),(c), and (d) are satisfied when we define $|f(p)|$ according to \eqref{eq:optimalfp} then
\[
\frac{M_2(R,T)}{M_1(R,T)} \geq 2V
\]
and 
\[
M_3(R,T) \ll T \exp\left(\frac{125 V^2 }{\sum_{p \in \mathcal{P}}\hat{\omega}\left(\tfrac{\log p}{\lambda}\right)^2 \frac{1}{p}}\right).
\]
Hence it is advantageous to select $P$ as small as our restrictions allow and $Q$ as large as our restrictions allow so that the upper bound on $M_3(R,T)$ is as strong as possible. We will make the following choice of parameters,
\begin{align*}
P &\coloneqq 25 V^2, \\
Q &\coloneqq e^{\lambda}.
\end{align*} 

We must now check that the restrictions are satisfied. For restrictions (b) and (c) there is nothing to check. For restriction (a), we first note that a routine calculation using the prime number theorem implies that
\begin{align*}
\sum_{p \in \mathcal{P}}\hat{\omega}\left(\tfrac{\log p}{\lambda}\right)^2 \frac{1}{p} &= \log \frac{\lambda}{\log 25V^2}-\frac{3}{2}+2 \frac{\log25 V^2}{\lambda} -\frac{1}{2} \left(\frac{\log 25V^2}{\lambda}\right)^2+o_{T\to \infty}\left(1\right) \\
&= \log \frac{\lambda}{\log V} - \log 2 - \frac{3}{2} + 4 \frac{\log V}{\lambda} - 2 \left(\frac{\log V}{\lambda}\right)^2 + o_{T\to\infty}(1).
\end{align*}
We can conclude that if the ratio $\frac{\lambda}{\log V}$ is large enough then this implies
\begin{equation} \label{eq:Psumbd}
\sum_{p \in \mathcal{P}}\hat{\omega}\left(\tfrac{\log p}{\lambda}\right)^2 \frac{1}{p} \geq 1
\end{equation}
and also
\begin{equation} \label{eq:Psumbd2}
\sum_{p \in \mathcal{P}}\hat{\omega}\left(\tfrac{\log p}{\lambda}\right)^2 \frac{1}{p} \geq \frac{1}{2}\log \frac{\lambda}{\log V}
\end{equation}
for all $T$ larger than some absolute constant. By the hypotheses of the lemma, $\lambda \geq 40 \log\log T$ and $\log V \leq \frac{1}{2}\log \log T$. Hence $\frac{\lambda}{\log V} \geq 80$ which one can verify is sufficiently large.

Applying \eqref{eq:Psumbd} and the fact that $\hat{\omega} \leq 1$ to our choice of $|f(p)|$ we see that
\[
|f(p)| \leq \frac{5V}{\sqrt{p}} \leq \frac{5V}{\sqrt{P}} \leq 1
\]
for any $p \in \mathcal{P}$. Hence, restriction (a) holds. 

Lastly we must check restriction (d). Note that \eqref{eq:Psumbd} implies $\sum_{p} |f(p)|^2 \leq 25V^2$. Using this we see that the left-hand side of restriction (d) is
\begin{align*}
-\frac{\log N}{\log Q} + \frac{2}{\log Q} \sum_{p \in \mathcal{P}} (\log p) |f(p)|^2
&= -\frac{1}{4}\frac{\log T}{\lambda} + \frac{2}{\lambda}\sum_{p \in \mathcal{P}} (\log p) |f(p)|^2  \\
&\leq -\frac{1}{4}\frac{\log T}{\lambda} + 2 \sum_{p \in \mathcal{P}} |f(p)|^2 \\
&\leq  -\frac{1}{4}\frac{\log T}{\lambda} + 50V^2\\
&\leq -\frac{1}{4}\frac{\log T}{\lambda} + \frac{1}{8}\frac{\log T}{\lambda} \\
&\leq -5.
\end{align*}
The second to last inequality follows from the hypothesis that $V \leq \frac{1}{20}\left(\frac{\log T}{\lambda}\right)^\frac{1}{2}$, and the last inequality follows from the hypothesis that $\lambda \leq \frac{1}{40}\log T$. Hence restriction (d) is satisfied.

So we have now seen that restrictions (a),(b),(c),(d) all hold. As we have seen, this means $M_2(R,T)/M_1(R,T) \geq 2V$ and
\[
M_3(R,T) \ll T \exp\left(\frac{125 V^2}{\sum_{p \in \mathcal{P}}\hat{\omega}\left(\tfrac{\log p}{\lambda}\right)^2 \frac{1}{p}}\right) \leq T \exp\left(\frac{250 V^2}{\log\frac{\lambda}{\log V}}\right)
\]
where the second inequality follows from \eqref{eq:Psumbd2}.

We may then conclude from \eqref{eq:FVbd} that
\begin{align*}
\meas F_V &\gg \frac{M_2(R,T)^4}{M_3(R,T)^2 M_4(T)} \\
&\gg T \exp\left(-500\frac{V^2}{\log\frac{\lambda}{\log V}}\right)\, V^4 (\log \lambda)^{-2} \\
&\gg T \exp\left(-500 \frac{V^2}{\log\frac{\lambda}{\log V}}\right).
\end{align*}
In the second inequality we have used our integral estimates and that $M_2(R,T) \gg V M_1(R,T) \gg V T$. In the last inequality we have used the hypothesis that $V \geq \sqrt{\log \lambda}$. This completes the proof of the lemma.

\section{Proof of Theorem \ref{Sthm}} \label{sec:finalSthm}
We will now prove Theorem~1 by combining the convolution formula with the measure bounds proved in the previous two sections. Suppose that we are given $V$ satisfying the bounds $\sqrt{\log \log T} \leq V \leq \kappa \left(\frac{\log T}{\log\log T}\right)^\frac{1}{3}$ where $\kappa$ is a small constant to be chosen later. We assume $T$ is large, and hence $V$ is large as well.

To apply the lemmas we will be selecting a value of the parameter $\lambda$ depending on $V$ and $T$. Before we do this, let us reestablish some of the notation from the lemmas. Recall that the convolution formula (i.e. Lemma~\ref{Sconvlemma}) states that
\[
\int_{-\log T}^{\log T} S(t+u) K_\lambda(u)\, du = \frac{1}{\pi}\Im D_\lambda(t) + Z_\lambda(t) + O(1).
\]
The contribution of the $Z_\lambda(t)$ term was discussed in Lemma~\ref{EVlemma}. In particular, as long as $\lambda$ and $V$ satisfy the necessary hypotheses, the lemma gives an upper bound on the measure of
\[
E_V = \{t \in [T,2T] \colon |Z_\lambda(t)| > V\}.
\]
Similarly, the contribution of the $\frac{1}{\pi}\Im D_\lambda(t)$ term is handled by Lemma~\ref{FVlemma}. Taking $\theta = \frac{3\pi}{2}$ in the lemma, we can get a lower bound on the measure of the set
\[
F_{4\pi V} =  \{t \in [T,2T] \colon \Re\left(e^{\frac{3 \pi}{2}i} D_\lambda(t)\right) > 4 \pi V\} = \{t \in [T,2T] \colon \tfrac{1}{\pi} \Im D_\lambda(t) > 4 V\}.
\]
Let us now define another set $G_{2V}$ as 
\[
G_{2V} \coloneqq \left\{ t \in [T,2T] \colon \int_{-\log T}^{\log T} S(t+u) K_\lambda(u)\, du > 2V \right\}.
\]
The convolution formula implies that $G_{2V} \supseteq F_{4\pi V} \setminus E_V$, so combining the measure bounds discussed above will yield a lower bound on the measure of $G_{2V}$.

We now select the value of the $\lambda$ parameter. We will consider two cases depending on whether $V > (\log T)^{\frac{2}{9}}$ or $V \leq (\log T)^{\frac{2}{9}}$. In the first case, let $\lambda = (\log T)/V^3$. This implies $\kappa^{-3} \log\log T \leq \lambda < (\log T)^\frac{1}{3}$ and one can verify that the hypotheses of our lemmas are satisfied as long as $\kappa$ is somewhat small. Applying Lemma~\ref{EVlemma} and Lemma~\ref{FVlemma} gives
\begin{align*}
&\meas E_V \ll T \exp\left(10 (\log T)^\frac{1}{3}\right) \exp(-d' V^2) \ll T \exp\left(-\frac{d'}{2} V^2\right), \text{ and} \\
&\meas F_{4\pi V} \gg T \exp\left(-500 \frac{(4\pi V)^2}{\log \frac{\log T}{V^3 \log 4\pi V}} \right) \gg T \exp\left(-D' \frac{V^2}{\log \frac{\log T}{V^3 \log V}} \right)
\end{align*}
where $D' > 0$ is a large absolute constant. By taking $\kappa$ very small, we can ensure that the quantity $\log \frac{\log T}{V^3 \log V}$ is uniformly larger than say $10D'/d'$.  In particular, this ensures that the measure of $F_{4\pi V}$ is at least twice the measure of $E_V$. 

Now consider the second case where $\sqrt{\log\log T} \leq V \leq (\log T)^{\frac{2}{9}}$. Here we choose $\lambda = (\log T)^\frac{1}{10}$. It is again easy to verify that the hypotheses of Lemmas~\ref{EVlemma} and \ref{FVlemma} are satisfied, and they imply that
\begin{align*}
&\meas E_V \ll T \exp\left(10 (\log T)^\frac{1}{10}\right) \exp(-d' V^\frac{1}{2} (\log T)^\frac{9}{20}) \ll T \exp\left(- (\log T)^\frac{9}{20}\right), \text{ and} \\
&\meas F_{4\pi V} \gg T \exp\left(-500 \frac{(4 \pi V)^2}{\log \frac{\log T}{10 \log 4 \pi V}} \right) \gg T \exp\left(-D'\frac{V^2}{\log \frac{\log T}{V^3 \log V}} \right).
\end{align*}
Once again we have that the measure of $F_{4\pi V}$ is at least twice that of $E_V$ for sufficiently large $T$ since $\meas F_{4\pi V} \gg T \exp(-(\log T)^\frac{4}{9})$ in this case and $\tfrac{4}{9} < \tfrac{9}{20}.$

Putting together the two cases, we have just confirmed that for any $V$ satisfying the hypotheses of the theorem it is possible to select $\lambda$ such that
\begin{equation} \label{eq:G2Vbd}
\meas G_{2V} \geq \meas F_{4\pi V} - \meas E_V \gg \meas F_{4\pi V} \gg T \exp\left(-D'\frac{V^2}{\log \frac{\log T}{V^3 \log V}} \right).
\end{equation}
This is nearly the bound stated in the conclusion of the theorem, but $G_{2V}$ consists of points in $[T,2T]$ where a \emph{smoothing} of $S(t)$ is greater than $2V$. It remains to convert this into a lower bound on the measure of the set where $S(t)$ is greater than $V$.

To pass from \eqref{eq:G2Vbd} to the conclusion of the theorem, first note that
\[
\int_{G_{2V}} \int_{-\log T}^{\log T} S(t+u) K_\lambda(u)\, du\, dt \geq 2V\, \meas G_{2V}
\]
by the definition of $G_{2V}$. The double integral on the left-hand side can be rewritten (by a change of variables and Fubini's theorem) as an integral of $S(u)$ times a smoothing of the indicator function of $G_{2V}$. Indeed, letting $1_A(u)$ denote the indicator function of a set $A$ and letting $\widetilde{K_\lambda}(u) \coloneqq K_\lambda(u) 1_{[-\log T, \log T]}(u)$, one quickly verifies that
\[
\int_{G_{2V}} \int_{-\log T}^{\log T} S(t+u) K_\lambda(u)\, du dt = \int S(u) (1_{G_{2V}} * \widetilde{K_\lambda})(u)\, du.
\]
Note that the support of the convolution $1_{G_{2V}} * \widetilde{K_\lambda}$ is contained in the interval $[T-\log T, 2T+\log T]$ so we can assume the last integral is taken over this bounded interval.

Letting $H_V \coloneqq \{ t \in [T-\log T,2T + \log T] \colon S(t) > V\}$, and using the same Paley-Zygmund inequality idea as in the previous section we have
\begin{align*}
\int S(u) (1_{G_{2V}} * \widetilde{K_\lambda})(u)\, du &\leq  V \int (1_{G_{2V}} * \widetilde{K_\lambda})(u)\, du + \int_{H_V} S(u) (1_{G_{2V}} * \widetilde{K_\lambda})(u)\, du \\
&\leq V\, ||1_{G_{2V}} * \widetilde{K_\lambda}||_1 + ||1_{G_{2V}} * \widetilde{K_\lambda}||_\infty \int_{H_V} S(u)\, du \\
&\leq V\, \meas G_{2V} + (\meas H_V)^\frac{1}{2} \left(\int_{T-\log T}^{2T+\log T} |S(u)|^2\, du\right)^\frac{1}{2}.
\end{align*}
The last inequality follows from the fact that $||1_{G_{2V}} * \widetilde{K_\lambda}||_1 \leq ||1_{G_{2V}}||_1 ||\widetilde{K_\lambda}||_1 \leq \meas G_{2V}$ and $||1_{G_{2V}} * \widetilde{K_\lambda}||_\infty \leq ||1_{G_{2V}}||_\infty ||\widetilde{K_\lambda}||_1 \leq 1$, and from applying Cauchy's inequality to the integral. Since the left-hand side of this string of inequalities is $\geq 2V\, \meas G_{2V}$ by our observations in the previous paragraph, rearranging we see that
\begin{equation} \label{eq:HVbd}
\meas H_V \geq V^2 \left(\meas G_{2V}\right)^2  \left(\int_{T-\log T}^{2T+\log T} |S(u)|^2\, du\right)^{-1}.
\end{equation}
The moments of $S(t)$ were estimated unconditionally by Selberg in \cite[Thm. 6]{selberg1946}. In particular, we have
\[
\int_{T-\log T}^{2T+\log T} |S(t)|^2 \, dt \ll T \log\log T.
\]
Inserting this and the lower bound \eqref{eq:G2Vbd} into \eqref{eq:HVbd} we conclude that
\begin{align*}
\meas H_V &\gg V^2 T^2 \exp\left(-2 D'\frac{V^2}{\log \frac{\log T}{V^3 \log V}} \right) \left(T \log\log T\right)^{-1} \\
&\gg T \exp\left(-D \frac{V^2}{\log \frac{\log T}{V^3 \log V}} \right)
\end{align*}
for an absolute constant $D > 0$. In the last inequality we used the assumption that $V \geq \sqrt{\log \log T}$. The intervals $[T-\log T, T)$ and $(2T, 2T+\log T]$ have very small measure compared to this lower bound, so they can safely be thrown away. This gives the conclusion stated in the theorem for positive values of $S(t)$ (after possibly adjusting the constant $D$ to turn the $\gg$ bound we have just proved into a $\geq$ bound). To get the same conclusion for $-S(t)$ the proof is identical except that we select $\theta = \frac{\pi}{2}$ in the application of Lemma~\ref{FVlemma}.

\section{Proof of Theorem \ref{Sthm2}} \label{sec:finalSthm2}
We now assume the Riemann hypothesis, and we shall indicate how the argument in the previous section can be modified to prove Theorem 2. Suppose $T$ is large and suppose $\sqrt{\log \log T} \leq V \leq \widetilde{\kappa} \left(\frac{\log T}{\log\log T}\right)^\frac{1}{2}$ where $\widetilde{\kappa}$ is a small constant to be chosen later.

Let $E_V$, $F_{4\pi V}$, and $G_{2V}$ denote the same sets as in the previous section. Because we have assumed the Riemann hypothesis, the $Z_\lambda(t)$ term in the convolution formula is always zero, so the set $E_V$ is empty. Consequently $G_{2V} \supseteq F_{4\pi V}$. As in the proof of Theorem 1, we must make a choice for the parameter $\lambda$ and then apply Lemma~\ref{FVlemma} to get a lower bound on the measure of $F_{4\pi V}$. In this case we select $\lambda = (\log T)/(80\pi V)^2$. One may verify that the hypotheses of Lemma~\ref{FVlemma} are satisfied as long as $\widetilde{\kappa}$ is rather small. In particular note that the constant $80\pi$ is chosen so that $4\pi V \leq \frac{1}{20}\left(\frac{\log T}{\lambda}\right)^{\frac{1}{2}}$ is satisfied. Hence we find
\[
\meas G_{2V} \geq \meas F_{4\pi V} \gg T \exp\left(-500 \frac{(4\pi V)^2}{\log \frac{\log T}{(80 \pi V)^2 \log 4\pi V}} \right) \gg T \exp\left(-\widetilde{D}' \frac{V^2}{\log \frac{\log T}{V^2 \log V}} \right)
\]
where $\widetilde{D}'>0$ is a large absolute constant. One can then repeat the argument in the previous section to pass from this lower bound on $\meas G_{2V}$ to the bound stated in Theorem~2.

\bibliographystyle{amsplain}
\bibliography{references}

\providecommand{\bysame}{\leavevmode\hbox to3em{\hrulefill}\thinspace}
\providecommand{\MR}{\relax\ifhmode\unskip\space\fi MR }
\providecommand{\MRhref}[2]{%
  \href{http://www.ams.org/mathscinet-getitem?mr=#1}{#2}
}
\providecommand{\href}[2]{#2}
\begin{thebibliography}{10}

\bibitem{arguin2023b}
L.-P. Arguin and E.~Bailey, \emph{Large deviation estimates of {S}elberg's
  {C}entral {L}imit {T}heorem and applications}, Int. Math. Res. Not. IMRN
  (2023), no.~23, 20574--20612. \MR{4675078}

\bibitem{bondarenko2018s}
A.~Bondarenko and K.~Seip, \emph{Extreme values of the {R}iemann zeta function
  and its argument}, Math. Ann. \textbf{372} (2018), no.~3-4, 999--1015.
  \MR{3880290}

\bibitem{chirre2021m}
A.~Chirre and K.~Mahatab, \emph{Large oscillations of the argument of the
  {R}iemann zeta-function}, Bull. Lond. Math. Soc. \textbf{53} (2021), no.~6,
  1776--1785. \MR{4379562}

\bibitem{farmer2007gh}
D.~W. Farmer, S.~M. Gonek, and C.~P. Hughes, \emph{The maximum size of
  {$L$}-functions}, J. Reine Angew. Math. \textbf{609} (2007), 215--236.
  \MR{2350784}

\bibitem{montgomery1977}
H.~L. Montgomery, \emph{Extreme values of the {R}iemann zeta function},
  Comment. Math. Helv. \textbf{52} (1977), no.~4, 511--518. \MR{460255}

\bibitem{radziwill2011}
M.~{Radziwi\l\l}, \emph{{Large deviations in Selberg's central limit theorem}},
  arXiv e-prints (2011), arXiv:1108.5092.

\bibitem{selberg1946}
A.~Selberg, \emph{Contributions to the theory of the {R}iemann zeta-function},
  Arch. Math. Naturvid. \textbf{48} (1946), no.~5, 89--155. \MR{20594}

\bibitem{soundararajan2008}
K.~Soundararajan, \emph{Extreme values of zeta and {$L$}-functions}, Math. Ann.
  \textbf{342} (2008), no.~2, 467--486. \MR{2425151}

\bibitem{soundararajan2009}
\bysame, \emph{Moments of the {R}iemann zeta function}, Ann. of Math. (2)
  \textbf{170} (2009), no.~2, 981--993. \MR{2552116}

\bibitem{titchmarsh}
E.~C. Titchmarsh, \emph{The theory of the {R}iemann zeta-function}, second ed.,
  The Clarendon Press, Oxford University Press, New York, 1986, Edited and with
  a preface by D. R. Heath-Brown. \MR{882550}

\bibitem{tsang1986}
K.~M. Tsang, \emph{Some {$\Omega$}-theorems for the {R}iemann zeta-function},
  Acta Arith. \textbf{46} (1986), no.~4, 369--395. \MR{871279}

\bibitem{tsang1993}
\bysame, \emph{The large values of the {R}iemann zeta-function}, Mathematika
  \textbf{40} (1993), no.~2, 203--214. \MR{1260885}

\end{thebibliography}
\end{document}